\documentclass{birkjour}

\usepackage{graphicx}

 \newtheorem{thm}{Theorem}[section]

 \theoremstyle{definition}
 
 \theoremstyle{remark}

 \numberwithin{equation}{section}

\begin{document}

\title[Dual complex Pell quaternions]{
\\ \\ 
Dual complex Pell quaternions}

\author[F\"{u}gen Torunbalc{\i} Ayd{\i}n]{F\"{u}gen Torunbalc{\i}  Ayd{\i}n*}
\address{%
Yildiz Technical University\\
Faculty of Chemical and Metallurgical Engineering\\
Department of Mathematical Engineering\\
Davutpasa Campus, 34220\\
Esenler, Istanbul,  TURKEY}

\email{faydin@yildiz.edu.tr ; ftorunay@gmail.com}

\thanks{*Corresponding Author}
\begingroup
    \renewcommand{\thefootnote}{}
    \footnotetext{%
      2010 AMS Subject Classification:
		   11B37, 20G20, 11R52.
	    } 
    \endgroup

\keywords{ Dual number; dual complex number; Pell number; dual complex Pell number; Pell quaternion; dual complex Pell quaternion. }

\begin{abstract}

In this paper, dual complex Pell numbers and quaternions are defined. Also, some algebraic properties of dual-complex Pell numbers and quaternions which are connected with dual complex numbers and Pell numbers are investigated. Furthermore, the Honsberger identity, Binet's formula, Cassini's identity, Catalan's identity for these quaternions are given. 
\end{abstract} 

\maketitle

\section{Introduction}

The real quaternions were first described by Irish mathematician William Rowan Hamilton in 1843. Hamilton \cite{A} introduced a set of real quaternions which can be represented as  
\begin{equation}\label{1}
H=\{\, q={q}_{0}+i\,{q}_{1}+j\,{q}_{2}+k\,{q}_{3}\,  |\, {q}_{0},\,{q}_{1},\,{q}_{2},\,{q}_{3}\in\mathbb{R}\,\} 
\end{equation}
where
\begin{equation*}
{i}^{2}={j}^{2}={k}^{2}=-1\,,\ \ i\ j=-j\ i=k\,,\quad j\ k=-k \ j=i\,,\quad k\ i=-i\ k=j\,.
\end{equation*}
\par The real quaternions constitute an extension of complex numbers into a four-dimensional space and can be considered as four-dimensional vectors, in the same way that complex numbers are considered as two-dimensional vectors.
Horadam \cite{B} defined complex Fibonacci and Lucas quaternions as follows
\begin{equation*}
Q_n=F_n+F_{n+1}\,i+F_{n+2}\,j+F_{n+3}\,k 
\end{equation*}
and
\begin{equation*}
K_n=L_n+L_{n+1}\,i+L_{n+2}\,j+L_{n+3}\,k
\end{equation*}
where $F_n$ and $L_n$ denote the $nth$ Fibonacci and Lucas numbers, respectively. Also, the imaginary quaternion units $i,\,j,\,\,k$ have the following rules
\begin{equation*}
i^2=j^2=k^2=-1\,,\ \ i\,j=-j\,i=k\,,\quad j\,k=-k \,j=i\,,\quad k\,i=-i\,k=j
\end{equation*}
The studies that follows is based on the work of Horadam \cite{C}-\cite{H}. 

In 1971,  Horadam studied  on the Pell and Pell-Lucas sequences and  he gave Cassini-like formula as follows \cite{I}:   
\begin{equation}\label{2}
{P}_{n+1}{P}_{n-1}-{P}_{n}^2=(-1)^n
\end{equation}
and Pell identities
\begin{equation}\label{3}
\left\{\begin{array}{l}
P_{r}\,P_{n+1}+P_{r-1}\,P_{n}=P_{n+r},\\
P_{n}(P_{n+1}+P_{n-1})=P_{2n},\\
P_{2n+1}+P_{2n}=2\,P_{n+1}^2-2\,P_{n}^2-(-1)^n,\\
P_{n}^2+P_{n+1}^2=P_{2n+1},\\
P_{n}^2+P_{n+3}^2=5(P_{n+1}^2+P_{n+2}^2),\\
P_{n+a}\,P_{n+b}-P_{n}\,P_{n+a+b}=(-1)^n\,P_{n}\,P_{n+a+b},\\
P_{-n}=(-1)^{n+1}\,P_{n}.
\end{array}\right.
\end{equation} 
and in 1985, Horadam and Mahon obtained Cassini-like formula as follows \cite{J}
\begin{equation}\label{4}
q_{n+1}\,q_{n-1}-q_{n}^2=8\,(-1)^{n+1}.
\end{equation} 
First the idea to consider Pell quaternions it was suggested by Horadam in paper \cite{K}.\\
\\
In 2016, \c{C}imen and \.{I}pek  introduced the Pell quaternions and the Pell-Lucas quaternions and gived properties of them \cite{L} as follows:
\begin{equation}\label{5}
{Q}{P_{n}}=\{{QP_{n}}=P_{n}\,e_0+P_{n+1}\,e_1+P_{n+2}\,e_2+P_{n+3}\,e_3 \, | \, P_{n}\,\,n\text{-}th\, Pell\, number\}
\end{equation}
where
\begin{equation*}
\begin{array}{rl}
{e_1}^2={e_2}^2={e_3}^2=-1,\,\,e_1\,e_2=-e_2\,e_1=e_3,\,\,e_2\,e_3=-e_3\,e_2=e_1,\, \\
e_3\,e_1=-e_1\,e_3=e_2.
\end{array}
\end{equation*}
In 2016, Anetta and Iwona introduced the Pell quaternions and the Pell octanions \cite{M} as follows:
\begin{equation}\label{6}
{R_{n}}={P}_{n}\,+ i\,{P}_{n+1}+ j\,{P}_{n+2}\,+\,k\,{P}_{n+3} 
\end{equation}
where
\begin{equation*}
{i}^{2}={j}^{2}={{k}^{2}}=i\,j\,k=-1\,,\ \ i\,j=-j\,i=k,\,j\,k=-k\,j=i,\,k\,i=-i\,k=j\,.
\end{equation*}
\\
In 2016, Torunbalc{\i} Ayd{\i}n and Y\"{u}ce introduced the dual Pell quaternions \cite{N} as follows:
\begin{equation}\label{7}
\begin{aligned}
{P_{D}}=\{{D}^P_{n}=&{{P}_{n}}+i\,{P_{n+1}}+j\,{P_{n+2}}+k\,{P_{n+3}}\,| \,\, {P}_{n}\,\,n\text{-}th\,\text{Pell number}\},
\end{aligned}
\end{equation}
where
\begin{equation*}
{i}^{2}={j}^{2}={k}^{2}=i\,j\,k=0\,,\ \ i\,j=\,-j\,i=\,j\,k=\,-k\,j=\,k\,i=\,-i\,k=0\,.
\end{equation*}
In 2017, Torunbalc{\i} Ayd{\i}n, K\"{o}kl\"{u} and Y\"{u}ce introduced generalized dual Pell quaternions \cite{O} as follows:
\begin{equation}\label{8}
\begin{aligned}
{Q}_{\mathbb{D}}=\{\bold{\mathbb{D}^P}_{n}&={\mathbb{P}_{n}}+i\,{\mathbb{P}_{n+1}}+j\,\,{\mathbb{P}_{n+2}}+k\,{\mathbb{P}_{n+3}}\,\left.\right| {\mathbb{P}_{n}},\, n\text{-}th\,\text{Gen. Pell number}\},
\end{aligned}
\end{equation}
where
\begin{equation*}
{i}^{2}={j}^{2}={k}^{2}=i\,j\,k=0\,,\ \ i\,j=-j\,i=j\,k=-k\,j=k\,i=-i\,k=0.
\end{equation*} 
Furthermore, Torunbalc{\i} Ayd{\i}n, K\"{o}kl\"{u} introduced the generalizations of the Pell sequence in 2017 \cite{P} as follows:
\begin{equation}\label{9}
\left\{\begin{array}{rl}
{\mathbb{P}_{0}}=&{{q}},\,{\mathbb{P}_{1}}=p,\,{\mathbb{P}_{2}}=2p+q, \,\, {p}\,{q}  \in{\mathbb Z} \\
{\mathbb{P}_{n}}=&2{\mathbb{P}_{n-1}}+\,{\mathbb{P}_{n-2}},\,\ n\geq 2 \\
or \\
{\mathbb{P}_n}=&(p-2q)P_{n}+q\,P_{n+1}=p\,P_{n}+q\,P_{n-1}
\end{array}\right.
\end{equation} 
In 2017, Toke{\c{s}}er, {\"U}nal and Bilgici, introduced split Pell quaternions \cite{Q} as follows: 
\begin{equation}\label{10}
{SP_n}={P}_{n}+i\,{P}_{n+1}+j\,{P}_{n+2}+k\,{P}_{n+3}.
\end{equation} 
where
\begin{equation*}
{i}^{2}=-1,\,\,{j}^{2}={k}^{2}=1,\,\,i\ j=-j\,i=k,\,\, j\,k=-j\,k=-i,\, k\,i=-i\,k=j.
\end{equation*} 
In 2017, Catarino and Vasco introduced dual k-Pell quaternions and Octonions \cite{R} as follows:
\begin{equation}\label{11}
\widetilde{{R}_{k,n}}=\widetilde{{P}_{k,n}}\,{e_0}+\widetilde{{P}_{k,n+1}}\,{e_1}+\widetilde{{P}_{k,n+2}}\,{e_2}+\widetilde{{P}_{k,n+3}}\,{e_3},
\end{equation} 
where
${\widetilde{{P_{k,n}}}=P_{k,n}+\varepsilon\,P_{k,n+1}}$,\,\, $P_{k,n}=2\,P_{k,n-1}+k\,P_{k,n-2}$,\,\, \,$n\ge2$
\begin{equation*}
{e_0}=1,\,{e_i}^{2}=-1,\,e_i\,e_j=-e_j\,e_i,\,\,i,\,j=1,\,2,\,3,
\end{equation*}
\begin{equation*}
\varepsilon\,\ne0,\,\,0\,\varepsilon=\varepsilon\,0=0,\,\,1\,\varepsilon=\varepsilon\,1=\varepsilon,\,\,\varepsilon^2=0.
\end{equation*} 
In 2018, Torunbalc{\i} Ayd{\i}n introduced bicomplex Pell and Pell-Lucas numbers \cite{S} as follows: 
\begin{equation}\label{12}
{BP_n}={P}_{n}+i\,{P}_{n+1}+j\,{P}_{n+2}+i\,j\,{P}_{n+3}
\end{equation}
and
\begin{equation}\label{13}
{BPL_n}={Q}_{n}+i\,{Q}_{n+1}+j\,{Q}_{n+2}+i\,j\,{Q}_{n+3}
\end{equation} 
where
${i}^{2}=-1,\,\,{j}^{2}=-1,\,\,i\ j=j\,i$.
\par In the 19 th century Clifford invented a new number system by using the notation $(\varepsilon)^2=0,\,\varepsilon \neq 0$. This number system was called dual number system and the dual numbers were represented in the form $A=a+\varepsilon\,a^*$ with $a,\,a^*\in \mathbb R$ \,\cite{T}. Afterwards, Kotelnikov (1895)  and Study (1903) generalized first applications of dual numbers to mechanics \cite{U}, \cite{V}. Besides mechanics, this concept has lots of applications in different areas such as algebraic geometry, kinematics, quaternionic formulation of motion in the theory of relativity. Majernik has introduced the multi-component number system \cite{W}. There are three types of the four-component number systems which have been constructed by joining the complex, binary and dual two-component numbers. Later, Farid Messelmi has defined the algebraic properties of the dual-complex numbers in the light of this study \cite{X}. There are many applications for  the theory of dual-complex numbers. In 2017, G\"{u}ngor and Azak have defined dual-complex Fibonacci and dual-complex Lucas numbers and their properties \cite{Y}.
Dual-complex numbers \cite{X} $w$ can be expressed in the form as 
\begin{equation}\label{14}
\mathbb{DC}=\{\, w={z}_{1}+\varepsilon {z}_{2} \, |\, {z}_{1},\,{z}_{2}\,\in\mathbb{C} \,\, \text{where} \, \, \, \varepsilon^2=0,\, \varepsilon \neq 0 \}.
\end{equation}
Here if $z_1=x_1+i\,x_2$ and $z_2=y_1+i\,y_2$, then any dual-complex number can be written 
\begin{equation}\label{15}
w=x_1+ix_2+\varepsilon\,y_1+i\,\varepsilon \,y_2
\end{equation}
\begin{equation*}
i^2=-1,\,\,\varepsilon\neq 0,\,\,\varepsilon^2=0,\,\,\,(i\,\varepsilon)^2=0.  
\end{equation*}
Addition, substraction and multiplication of any two dual-complex numbers $w_1$ and $w_2$ are defined by
\begin{equation}\label{16}
\begin{aligned}
w_1\pm w_2=&(z_1+\varepsilon z_2)\pm (z_3+\varepsilon z_4)=(z_1\pm z_3)+\varepsilon(z_2\pm z_4), \\
w_1\times\,w_2=&(z_1+\varepsilon z_2)\times\,(z_3+\varepsilon z_4)=z_1\,z_3+\varepsilon\,(z_2\,z_4+z_2\,z_3).
\end{aligned}
\end{equation}
On the other hand, the division of two dual-complex numbers are given by
\begin{equation}\label{17}
\begin{aligned}
&\frac{w_1}{w_2}=\frac{z_1+\varepsilon z_2}{z_3+\varepsilon z_4} \\
&\frac{(z_1+\varepsilon z_2)(z_3-\varepsilon z_4)}{(z_3+\varepsilon z_4)(z_3-\varepsilon z_4)}=
\frac{z_1}{z_3}+\varepsilon \,\frac{z_2\,z_3-z_1\,z_4}{z_3^2}.
\end{aligned}
\end{equation}
If $Re(w_2)\neq 0$,then the division $\frac{w_1}{w_2}$ is possible. The dual-complex numbers are defined by the basis $\{1,i,\varepsilon,i\,\varepsilon \}$. Therefore, dual-complex numbers, just like quaternions, are a generalization of complex numbers by means of entities specified by four-component numbers. But real and dual quaternions are non commutative, whereas, dual-complex numbers are commutative. The real and dual quaternions form a division algebra, but dual-complex numbers form a commutative ring with characteristics $0$. Moreover, the multiplication of these numbers gives the dual-complex numbers the structure of 2-dimensional complex Clifford Algebra and 4-dimensional real Clifford Algebra.  
The base elements of the dual-complex numbers satisfy the following commutative multiplication scheme (Table 1).
\begin{table}[]
\centering
\caption{Multiplication scheme of dual-complex numbers}
\begin{tabular}{c rrrr}
\hline
 $x$&  $1$&  $i$& $\varepsilon$& $i\,\varepsilon$\cr  
\hline
 $1$&  $1$&  $i$&  $\varepsilon$& $i\,\varepsilon$\cr  
 $i$&  $i$&  $-1$&  $i\,\varepsilon$& $-\varepsilon$\cr 
$\varepsilon$&  $\varepsilon$&  $i\,\varepsilon$&  $0$& $0$\cr
$i\,\varepsilon$& $i\,\varepsilon$& $-\varepsilon$& $0$& $0$\cr 
\hline 
\end{tabular}
\end{table}
Five different conjugations can operate on dual-complex numbers \cite{X} as follows: 
\begin{equation}\label{18}
\begin{aligned}
w=&x_1+ix_2+\varepsilon\,y_1+i\,\varepsilon y_2,\cr 
{w}^{*_1}=&(x_1-ix_2)+\varepsilon (y_1-i\,y_2)=(z_1)^*+\varepsilon\,(z_2)^*,\cr
{w}^{*_2}=&(x_1+i\,x_2)-\varepsilon \,(y_1+i\,y_2)=z_1-\varepsilon\,z_2,\cr
{w}^{*_3}=&(x_1-i\,x_2)-\varepsilon \,(y_1-i\,y_2)=z_1^*-\varepsilon\,z_2^*, \cr
{w}^{*_4}=&(x_1-i\,x_2)(1-\varepsilon \,\frac{y_1+i\,y_2 }{x_1+ix_2} \,)=(z_1)^*(1-\varepsilon\frac{z_2}{z_1}), \cr
{w}^{*_5}=&(y_1+i\,y_2)-\varepsilon (x_1+i\,x_2)=z_2-\varepsilon z_1. 
\end{aligned}
\end{equation}
Therefore, the norm of the dual-complex numbers is defined as 
\begin{equation}\label{19}
\begin{aligned}
{{N}_{w}^{*_1}}=&\left\| {w\times{w}^{*_1}} \right\|=\sqrt{\left|{z}_{1}\right|^2+2\,\varepsilon Re({z}_{1}\,{z_2}^*)}, \\ 
{{N}_{w}^{*_2}}=&\left\| {w\times{w}^{*_2}} \right\|=\sqrt{{z}_{1}^2}, \\ 
{{N}_{w}^{*_3}}=&\left\| {w\times{w}^{*_3}} \right\|=\sqrt{\left|{z}_{1}\right|^2-2\,i\,\varepsilon Im({z}_{1}\,{z_2}^*)},\\
{{N}_{w}^{*_4}}=&\left\| {w\times{w}^{*_4}} \right\|=\sqrt{\left|{z}_{1}\right|^2},\\
{{N}_{w}^{*_5}}=&\left\| {w\times{w}^{*_5}} \right\|=\sqrt{{z}_{1}\,{z}_{2}+\varepsilon ({z}_{2}^2-{z_1}^2)}. 
\end{aligned}
\end{equation} 
\par In this paper, the dual-complex Pell numbers and quaternions will be defined. The aim of this work is to present in a unified manner a variety of algebraic properties of the dual-complex Pell quaternions as well as both the dual-complex numbers and dual-complex Pell numbers. In particular, using five types of conjugations, all the properties established for dual-complex numbers and dual-complex Pell numbers are also given for the dual-complex Pell quaternions. In addition, the Honsberger identity, the d'Ocagne's identity, Binet's formula, Cassini's identity, Catalan's identity for these quaternions are given.

\section{The dual-complex Pell numbers}
The dual-complex Pell and Pell-Lucas numbers can be defined by the basis $\{1,\,i,\,\varepsilon,\,i\,\varepsilon\,\}$, where $i$,\,$j$ \,and\, $i\,j$ satisfy the conditions 
\begin{equation*}
i^2=-1,\,\,\varepsilon\neq 0,\,\,\varepsilon^2=0,\,\,\,(i\,\varepsilon)^2=0.  
\end{equation*}
as follows
\begin{equation}\label{20}
\begin{aligned}
{\mathbb{DC}P_n}=&({P}_{n}+i\,{P}_{n+1})+\varepsilon \,({P}_{n+2}+i\,{P}_{n+3}) \\
=& {P}_{n}+i\,{P}_{n+1}+\varepsilon \,{P}_{n+2}+i\,\varepsilon\,{P}_{n+3}
\end{aligned}
\end{equation}
and
\begin{equation}\label{21}
\begin{aligned}
{\mathbb{DC}PL_n}=&({PL}_{n}+i\,{PL}_{n+1})+\varepsilon \,({PL}_{n+2}+i\,{PL}_{n+3}) \\
=& {PL}_{n}+i\,{PL}_{n+1}+\varepsilon \,{PL}_{n+2}+i\,\varepsilon\,{PL}_{n+3}.
\end{aligned}
\end{equation}
With the addition, substraction and multiplication by real scalars of two dual-complex Pell numbers, the dual-complex Pell number can be obtained again. 
Then, the addition and subtraction of the dual-complex Pell numbers are defined by 
\begin{equation}\label{22}
\begin{array}{rl}
{\mathbb{DC}P_n}\pm{\mathbb{DC}P_m}=&({P}_{n}\pm{P}_{m})+i\,({P}_{n+1}\pm{P}_{m+1})+\varepsilon \,({P}_{n+2}\pm{P}_{m+2}) \\
&+i\,\varepsilon\,({P}_{n+3}\pm{P}_{m+3}). 
\end{array}
\end{equation}
The multiplication of a dual-complex Pell number by the real scalar $\lambda$ is defined as 
\begin{equation}\label{23}
{\lambda}\,{\mathbb{DC}P_n}=\lambda\,P_n+i\,\lambda\,{P}_{n+1}+\varepsilon\,\lambda\,{P}_{n+2}+i\,\varepsilon\,\lambda\,\,{P}_{n+3}.
\end{equation} 
By using (Table 1) the multiplication of two dual-complex Pell numbers is defined by
\begin{equation}\label{24}
\begin{array}{rl}
{\mathbb{DC}P_n}\times\,{\mathbb{DC}P_m}=&({P}_{n}\,{P}_{m}-{P}_{n+1}\,{P}_{m+1})+i\,({P}_{n+1}\,{P}_{m}+{P}_{n}\,{P}_{m+1}) \\
&+\varepsilon \,({P}_{n}\,{P}_{m+2}-{P}_{n+1}\,{P}_{m+3}+{P}_{n+2}\,{P}_{m}-{P}_{n+3}\,{P}_{m+1}) \\
&+i\,\varepsilon\,({P}_{n+1}\,{P}_{m+2}+{P}_{n}\,{P}_{m+3}+{P}_{n+3}\,{P}_{m}+{P}_{n+2}\,{P}_{m+1}) \\
=&{\mathbb{DC}P_m}\times\,{\mathbb{DC}P_n}.
\end{array}
\end{equation} 
Also, there exits five conjugations as follows:
\begin{equation}\label{25}
\begin{aligned}
\mathbb{DC}{P}_{n}^{*_1}=&{P}_{n}-i\,{P}_{n+1}+\varepsilon\,{P}_{n+2}-i\,\varepsilon\,{P}_{n+3},\quad \text{complex-conjugation} 
\end{aligned}
\end{equation}
\begin{equation}\label{26}
\begin{aligned}
\mathbb{DC}{P}_{n}^{*_2}=&{P}_{n}+i\,{P}_{n+1}-\varepsilon\,{P}_{n+2}-i\,\varepsilon\,{P}_{n+3},\quad \text{dual-conjugation}  
\end{aligned}
\end{equation}
\begin{equation}\label{27}
\begin{aligned}
\mathbb{DC}{P}_{n}^{*_3}=&{P}_{n}-i\,{P}_{n+1}-\varepsilon\,{P}_{n+2}+i\,\varepsilon\,{P}_{n+3},\quad \text{coupled-conjugation} 
\end{aligned}
\end{equation}
\begin{equation}\label{28}
\begin{aligned}
\mathbb{DC}{P}_{n}^{*_4}=&({P}_{n}-i\,{P}_{n+1})\,.\,\varepsilon\,(1-\frac{{P}_{n+2}+i\,{P}_{n+3}}{{P}_{n}+i\,{P}_{n+1}}), \quad \text{dual-complex-conjugation} 
\end{aligned}
\end{equation}
\begin{equation}\label{29}
\begin{aligned}
\mathbb{DC}{P}_{n}^{*_5}=&{P}_{n+2}+i\,{P}_{n+3}-\varepsilon\,{P}_{n}-i\,\varepsilon\,{P}_{n+1}, \quad \text{anti-dual-conjugation}.
\end{aligned}
\end{equation}
\\
In this case, we can give the following relations:
\begin{equation}\label{30}
\mathbb{DC}{P}_{n}\,(\mathbb{DC}{P}_{n})^{*_1}=P_{2n+1}+2\,\varepsilon{P}_{2n+3}, 
\end{equation}
\begin{equation}\label{31}
\begin{aligned}
\mathbb{DC}{P}_{n}\,(\mathbb{DC}{P}_{n})^{*_2}=-q_{n}\,q_{n+1}+2\,i\,{P}_{n}\,P_{n+1},\quad {q_n} \quad \text{n-th}\, \text{mod.\,Pell\,number} 
\end{aligned}
\end{equation}
\begin{equation}\label{32}
\mathbb{DC}{P}_{n}\,(\mathbb{DC}{P}_{n})^{*_3}=P_{2n+1}-4\,\,i\,\varepsilon\,(-1)^{n}, 
\end{equation}
\begin{equation}\label{33}
\mathbb{DC}{P}_{n}\,(\mathbb{DC}{P}_{n})^{*_4}=P_{2n+1}, \\
\end{equation}
\begin{equation}\label{34}
\mathbb{DC}{P}_{n}+\,(\mathbb{DC}{P}_{n})^{*_1}=2\,(P_{n}+\,\varepsilon{P}_{n+2}),
\end{equation}
\begin{equation}\label{35}
\mathbb{DC}{P}_{n}+\,(\mathbb{DC}{P}_{n})^{*_2}=2\,(P_{n}+i\,{P}_{n+1}), 
\end{equation}
\begin{equation}\label{36}
\begin{aligned}
\mathbb{DC}{P}_{n}+\,(\mathbb{DC}{P}_{n})^{*_3}=2\,(P_{n}+i\,\varepsilon{P}_{n+3}),
\end{aligned}
\end{equation}
\begin{equation}\label{37}
\begin{aligned}
({P}_{n}+i\,{P}_{n+1})\,(\mathbb{DC}{P}_{n})^{*_4}=&(P_{2n+1}-\varepsilon{P}_{2n+3}+2\,i\,\varepsilon (-1)^n) \\
=&({P}_{n}-i\,{P}_{n+1})\,(\mathbb{DC}{P}_{n})^{*_2}, 
\end{aligned}
\end{equation}
\begin{equation}\label{38}
\begin{aligned}
\varepsilon\,\mathbb{DC}{P}_{n}+\,(\mathbb{DC}{P}_{n})^{*_5}=P_{n+2}+i\,{P}_{n+3},
\end{aligned}
\end{equation}
\begin{equation}\label{39}
\begin{aligned}
\mathbb{DC}{P}_{n}-\varepsilon\,(\mathbb{DC}{P}_{n})^{*_5}=P_{n}+i\,{P}_{n+1}.
\end{aligned}
\end{equation}
The norm of the dual-complex Pell numbers ${\,\mathbb{DC}{P}_{n}}$ is defined in five different ways as follows
\begin{equation}\label{40}
\begin{array}{rl}
{N}_{\mathbb{DC}{P}_{n}^{*_1}}=&\|\mathbb{DC}{P}_{n}\times\,(\mathbb{DC}{P}_{n})^{*_1}\|^2 \\
=&({P}_{n}^2+{{P}_{n+1}^2})+2\,\varepsilon(\,{P}_{n}\,{P}_{n+2}+{P}_{n+1}\,{P}_{n+3}) \\
=&{P}_{2n+1}+2\,\varepsilon{P}_{2n+3}, 
\end{array} 
\end{equation}
\begin{equation}\label{41}
\begin{array}{rl}
{N}_{\mathbb{DC}{P}_{n}^{*_2}}=&\|\mathbb{DC}{P}_{n}\times\,(\mathbb{DC}{P}_{n})^{*_2}\|^2 \\
=&|({P}_{n}^2-{P}_{n+1}^2)+2\,i\,{P}_{n}\,{P}_{n+1}\,| \\
=&|-{q}_{n}\,{q}_{n+1}+2\,i\,{P}_{n}\,{P}_{n+1}\,|,
\end{array} 
\end{equation}
\begin{equation}\label{42}
\begin{array}{rl}
{N}_{\mathbb{DC}{P}_{n}^{*_3}}=&\|\mathbb{DC}{P}_{n}\times\,(\mathbb{DC}{P}_{n})^{*_3}\|^2 \\
=&({P}_{n}^2+{P}_{n+1}^2)+2\,i\,\varepsilon({P}_{n}\,{P}_{n+3}-{P}_{n+1}\,{P}_{n+2}) \\
=&{P}_{2n+1}-4\,i\,\varepsilon (-1)^{n}.
\end{array}  
\end{equation}
\begin{equation}\label{43}
\begin{array}{rl}
{N}_{\mathbb{DC}{P}_{n}^{*_4}}=&\|\mathbb{DC}{P}_{n}\times\,(\mathbb{DC}{P}_{n})^{*_4}\|^2 \\
=&{P}_{n}^2+{P}_{n+1}^2={P}_{2n+1}.
\end{array}  
\end{equation}
\section{The dual-complex Pell and Pell-Lucas quaternions}
In this section, firstly the dual-complex Pell quaternions will be defined. The dual-complex Pell quaternions and the dual-complex Pell-Lucas quaternions are defined by using the dual-complex Pell numbers and the dual-complex Pell-Lucas numbers respectively, as follows 
\begin{equation}\label{44}
\begin{aligned}
\mathbb{DC}^{P_n}=\{{Q_P}_{n}=&{P}_{n}+i\,{P}_{n+1}+\varepsilon\,{P}_{n+2}+i\,\varepsilon\,{P}_{n+3}\,\left.\right|\,\,\, {P_{n}}\,,\, n-th,\,\text{Pell number}\} 
\end{aligned}
\end{equation}
and
\begin{equation}\label{45}
\begin{aligned}
\mathbb{DC}^{PL_n}=\{{Q_{PL}}_{n}=&{Q}_{n}+i\,{Q}_{n+1}+\varepsilon\,{Q}_{n+2}+i\,\varepsilon\,{Q}_{n+3}\,\left.\right|\,\,\, {Q}_{n}\,,\\
& n-th,\,\text{Pell-Lucas number}\} 
\end{aligned}
\end{equation}
where
\begin{equation*}
\begin{aligned}
{i}^2=-1,\,\varepsilon\neq 0,\, \,{\varepsilon}^{2}=0,\,\,\, (i\,\varepsilon)^2=0.
\end{aligned}
\end{equation*}
Let ${\,{{Q}_P}_{n}}$ and ${\,{{Q}_P}_{m}}$ be two dual-complex Pell quaternions such that
\begin{equation}\label{46} 
{Q_P}_{n}={P}_{n}+i\,{P}_{n+1}+\varepsilon\,{P}_{n+2}+i\,\varepsilon\,{P}_{n+3}
\end{equation}
and 
\begin{equation}\label{47}
{Q_P}_{m}={P}_{m}+i\,{P}_{m+1}+\varepsilon\,{P}_{m+2}+i\,\varepsilon\,{P}_{m+3}
\end{equation}
Then, the addition and subtraction of two dual-complex Pell quaternions are defined in the obvious way,  
\begin{equation}\label{48}
\begin{array}{rl}
{Q_P}_{n}\pm{\,{Q_P}_{m}}=&({P}_{n}+i\,{P}_{n+1}+\varepsilon\,{P}_{n+2}+i\,\varepsilon\,{P}_{n+3}) \\ &\pm ({P}_{m}+i\,{P}_{m+1}+\varepsilon\,{P}_{m+2}+i\,\varepsilon\,{P}_{m+3}) \\
=&({{P}_{n}}\pm{{P}_{m}})+i\,({P}_{n+1}\pm{P}_{m+1})+\varepsilon\,({P}_{n+2}\pm{P}_{m+2}) \\ &+ i\,\varepsilon\,({P}_{n+3}\pm{P}_{m+3}).
\end{array}
\end{equation}
Multiplication of two dual-complex Pell quaternions is defined by 
\begin{equation}\label{49}
\begin{array}{rl}
{Q_P}_{n}\times\,{Q_P}_{m}=&({P}_{n}+i\,{P}_{n+1}+\varepsilon\,{P}_{n+2}+i\,\varepsilon\,{P}_{n+3}) \\&\,({P}_{m}+i\,{P}_{m+1}+\varepsilon\,{P}_{m+2}+i\,\varepsilon\,{P}_{m+3}) \\
=&({P}_{n}{P}_{m}-{P}_{n+1}{P}_{m+1})\\ 
& +i\,({P}_{n+1}{P}_{m}+{P}_{n}{P}_{m+1})\\ 
&+\varepsilon\,({P}_{n}{P}_{m+2}-{P}_{n+1}{P}_{m+3}+{P}_{n+2}{P}_{m}-{P}_{n+3}{P}_{m+1}) \\ 
&+ i\,\varepsilon\,({P}_{n+1}{P}_{m+2}+{P}_{n}{P}_{m+3}+{P}_{n+3}{P}_{m}+{P}_{n+2}{P}_{m+1}) \\
=&{Q_P}_{m}\times\,{Q_P}_{n}\,.
\end{array}
\end{equation}
The scaler and the dual-complex vector parts of the dual-complex Pell quaternion $({{Q}_P}_{n})$ are denoted by 
\begin{equation}\label{50}
{S}_{Q_{P_n}}={P}_{n} \ \ \text{and} \ \ \ {V}_{Q_{P_n}}=i\,{P}_{n+1}+\varepsilon\,{P}_{n+2}+i\,\varepsilon\,{P}_{n+3}.	
\end{equation}
Thus, the dual-complex Pell quaternion ${\,{Q}_P}_{n}$  is given by ${\,{Q}_P}_{n}={S}_{Q_{P_n}}+{V}_{Q_{P_n}}$.            
\par The five types of conjugation given for the dual-complex Pell numbers are the same within the dual-complex Pell quaternions. Furthermore, the conjugation properties for these quaternions are given by the relations in (\ref{25})-(\ref{29}).
In the following theorem, some properties related to the dual-complex Pell quaternions are given. 
\begin{thm} 
Let ${\,{{Q}_P}_{n}}$  be the dual-complex Pell quaternion. In this case, we can give the following relations: 
\begin{equation}\label{51}
{{Q}_P}_{n}+2\,{{Q}_P}_{n+1}={{Q}_P}_{n+2}
\end{equation} \,
\begin{equation}\label{52}
{{Q}_{PL}}_{n}+2\,{{Q}_{PL}}_{n+1}={{Q}_{PL}}_{n+2}
\end{equation} \,
\begin{equation}\label{53}
{{Q}_P}_{n+1}+{{Q}_P}_{n-1}={{Q}_{PL}}_{n}
\end{equation} \,
\begin{equation}\label{54}
{{Q}_P}_{n+2}-{{Q}_P}_{n-2}=2\,{{Q}_{PL}}_{n}
\end{equation} \,
\begin{equation}\label{55}
{{Q}_P}_{n+2}+{{Q}_P}_{n-2}=6\,{{Q}_{P}}_{n}
\end{equation} \,
\begin{equation}\label{56}
\begin{array}{rl}
({{Q}_P}_{n})^2+({{Q}_P}_{n+1})^2=& {{Q}_P}_{2n+1}-2\,{P}_{2n+3}+i\,{P}_{2n+2} \\
&+\varepsilon\,(-11\,{P}_{2n+3}+2\,{P}_{2n+1} \\
&+4\,i\,\varepsilon\,(P_{2n+4}+{P}_{n+2}^2),
\end{array}
\end{equation} \,
\begin{equation}\label{57}
\begin{array}{rl}
({{Q}_P}_{n+1})^2-({{Q}_P}_{n-1})^2=& {{Q}_P}_{2n}+({P}_{2n}-2\,{P}_{2n+2})+3\,i\,{P}_{2n+1} \\
&-\varepsilon\,({P}_{2n+2}+8\,{P}_{2n+3}) \\
&+i\,\varepsilon\,(3\,{P}_{2n+3}-2\,{P}_{2n}+12\,{P}_{n+1}\,{P}_{n+2} \\
&-4\,{P}_{n+1}\,{P}_{n-1}),
\end{array}
\end{equation} \,
\begin{equation}\label{58}
{{Q}_P}_{n}-i\,({{Q}_P}_{n+1})^{*_3}-\varepsilon\,{{Q}_P}_{n+2}-i\,\varepsilon\,{{Q}_P}_{n+3}=2\,(-{P}_{n+1}+\varepsilon\,{P}_{n+4}).
\end{equation}
\end{thm}
\begin{proof} 
(\ref{51})-(\ref{55}): It is easily proved using (\ref{44}),(\ref{45}).  
(\ref{56}): By using (\ref{44}) we get,
\begin{equation*}
\begin{array}{rl}
({{Q}_P}_{n})^2+({{Q}_P}_{n+1})^2=&({P}_{2n+1}-{P}_{2n+3})+2\,i\,{P}_{2n+2} \\ 
&+2\,\varepsilon\,({P}_{2n+1}-5\,{P}_{2n+3})+2\,i\,\varepsilon\,(\frac{5}{2}\,{P}_{2n+4}+2\,{P}_{n+2}^2) \\ 
=& ({P}_{2n+1}+i\,{P}_{2n+2}+\varepsilon\,{P}_{2n+3}+i\,\varepsilon\,{P}_{2n+4})-{P}_{2n+3} \\
&+i\,{P}_{2n+2}-\varepsilon\,({P}_{2n+1}-11\,{P}_{2n+3})+4\,i\,\varepsilon\,(P_{2n+4}+{P}_{n+2}^2) \\
=&{{Q}_P}_{2n+1}-{P}_{2n+3}+i\,{P}_{2n+2}-\varepsilon\,({P}_{2n+1}-11\,{P}_{2n+3}) \\
&+4\,i\,\varepsilon\,({P}_{2n+4}+{P}_{n+2}^2). 
\end{array}
\end{equation*}
(\ref{57}): By using (\ref{44}) we get,
\begin{equation*}
\begin{array}{rl}
({{Q}_P}_{n+1})^2-({{Q}_P}_{n-1})^2=&2\,({P}_{2n}-{P}_{2n+2})+4\,i\,{P}_{2n+1}-8\,\varepsilon\,{P}_{2n+3} \\
&+2\,i\,\varepsilon\,(6\,{P}_{n+1}\,{P}_{n+2}-2\,{P}_{n+1}\,{P}_{n-1}+2\,{P}_{2n+3}-{P}_{2n}) \\ 
=& ({P}_{2n}+i\,{P}_{2n+1}+\varepsilon\,{P}_{2n+2}+i\,\varepsilon\,{P}_{2n+3}) \\
&+({P}_{2n}-2\,{P}_{2n+2})+3\,i\,{P}_{2n+1}-\varepsilon\,(8\,{P}_{2n+3}+{P}_{2n+2}) \\
&+i\,\varepsilon\,(3\,{P}_{2n+3}-2\,{P}_{2n}+12\,{P}_{n+1}\,{P}_{n+2}-4\,{P}_{n+1}\,{P}_{n-1})\\ 
=& {{Q}_P}_{2n}-({P}_{2n}-2\,{P}_{2n+2})+3\,i\,{P}_{2n+1} -\varepsilon\,(8\,{P}_{2n+3}+{P}_{2n+2}) \\
&+i\,\varepsilon\,(3\,{P}_{2n+3}-2\,{P}_{2n}+12\,{P}_{n+1}\,{P}_{n+2}-4\,{P}_{n+1}\,{P}_{n-1}) . 
\end{array}
\end{equation*}
(\ref{58}): By using (\ref{44}) and (\ref{27}) we get,
\begin{equation*}
\begin{array}{rl}
{{Q}_P}_{n}-i\,({{Q}_P}_{n+1})^{*_3}-\varepsilon\,{{Q}_P}_{n+2}-i\,\varepsilon\,{{Q}_P}_{n+3}=&({P}_{n}-{P}_{n+2})+2\,\varepsilon\,{P}_{n+4} \\
=&-2\,({P}_{n+1}+\varepsilon\,{P}_{n+4}).
\end{array}
\end{equation*}
\end{proof} 
\begin{thm} 
For \,$n,m\ge 0$ the Honsberger identity for  the dual-complex Pell quaternions ${Q_P}_{n}$ and ${Q_P}_{m}$ \, is given by
\begin{equation}\label{59}
\begin{array}{rl}
{Q_P}_{n}\,{Q_P}_{m}+{Q_P}_{n+1}\,{Q_P}_{m+1}=& {Q_P}_{n+m+1}-{P}_{n+m+3}+i\,{P}_{n+m+2} \\
&+\varepsilon\,({P}_{n+m+3}-2\,{P}_{n+m+5})+3\,i\,\varepsilon\,{P}_{n+m+4}.
\end{array}
\end{equation}
\end{thm}
\begin{proof}
(\ref{59}): By using (\ref{44}) we get,
\begin{equation*}
\begin{array}{rl}
{{Q}_P}_{n}\,{{Q}_P}_{m}+{{Q}_P}_{n+1}\,{{Q}_P}_{m+1}=&[({P}_{n}{P}_{m}-{P}_{n+2}{P}_{m+2})\\
\quad \quad \quad +&\,i\,[({P}_{n}{P}_{m+1}+{P}_{n+1}{P}_{m+2})+({P}_{n+1}{P}_{m}+{P}_{n+2}{P}_{m+1}) \\
\quad \quad \quad +&\varepsilon\,[({P}_{n}{P}_{m+2}+{P}_{n+1}{P}_{m+3})-({P}_{n+1}{P}_{m+3}+{P}_{n+2}{P}_{m+4}) \\ \quad \quad \quad +&({P}_{n+2}{P}_{m}+{P}_{n+3}{P}_{m+1})-({P}_{n+3}{P}_{m+1}+{P}_{n+4}{P}_{m+2})] \\
\quad \quad \quad +&i\,\varepsilon\,[({P}_{n}{P}_{m+3}+{P}_{n+1}{P}_{m+4})+({P}_{n+1}{P}_{m+2}+{P}_{n+2}{P}_{m+3}) \\ \quad \quad \quad +&({P}_{n+2}{P}_{m+1}+{P}_{n+3}{P}_{m+2})+({P}_{n+3}{P}_{m}+{P}_{n+4}{P}_{m+1})] \\
=&-2\,({P}_{n+m+2})+2\,i\,{P}_{n+m+2}+2\,\varepsilon\,({P}_{n+m+3}-{P}_{n+m+5}) \\
\quad \quad \quad +&\,4\,i\,\varepsilon\,{P}_{n+m+4} \\
=&{Q_P}_{n+m+1}-{P}_{n+m+3}+i\,{P}_{n+m+2} \\
\quad \quad \quad +& \varepsilon\,({P}_{n+m+3}-{P}_{n+m+5})+3\,i\,\varepsilon\,{P}_{n+m+4}.
\end{array}
\end{equation*} \\
where the identity ${P}_{n}{P}_{m}+{P}_{n+1}{P}_{m+1}={P}_{n+m+1}$ is used \,\cite{I,Z,Z1}. 
\end{proof}
\begin{thm} 
For $n,m\ge 0$ the d'Ocagne's identity for  the dual-complex Pell quaternions ${Q_P}_{n}$ and ${Q_P}_{m}$ is given by
\begin{equation}\label{60}
\begin{array}{rl}
{Q_P}_{m}\,{Q_P}_{n+1}-{Q_P}_{m+1}\,{Q_P}_{n}=&(-1)^{n}\,2\,{P}_{m-n}\,(1+i\,+6\varepsilon\,+6\,i\,\varepsilon\,).
\end{array}
\end{equation}
\end{thm}
\begin{proof}
(\ref{60}): By using (\ref{44}) we get,
\begin{equation*}
\begin{array}{rl}
{Q_P}_{m}\,{Q_P}_{n+1}-{Q_P}_{m+1}\,{Q_P}_{n}=&[\,({P}_{m}{P}_{n+1}-{P}_{m+1}{P}_{n})-({P}_{m+1}{P}_{n+2}-{P}_{m+2}{P}_{n+1})\,]\\
&+\,i\,[({P}_{m}{P}_{n+2}-{P}_{m+1}{P}_{n+1})+({P}_{m+1}{P}_{n+1}-{P}_{m+2}{P}_{n})\,] \\
&+\varepsilon\,[({P}_{m}{P}_{n+3}-{P}_{m+1}{P}_{n+2})-({P}_{m+1}{P}_{n+4}-{P}_{m+2}{P}_{n+3}) \\
&+({P}_{m+2}{P}_{n+1}-{P}_{m+3}{P}_{n})-({P}_{m+3}{P}_{n+2}-{P}_{m+4}{P}_{n+1})] \\
&+i\,\varepsilon\,[({P}_{m}{P}_{n+4}-{P}_{m+1}{P}_{n+3})+({P}_{m+1}{P}_{n+3}-{P}_{m+2}{P}_{n+2}) \\
&+({P}_{m+2}{P}_{n+2}-{P}_{m+3}{P}_{n+1})+({P}_{m+3}{P}_{n+1}-{P}_{m+4}{P}_{n})] \\
=&(-1)^n\,2\,({P}_{m-n})+i\,(-1)^n\,({P}_{m-n+1}-{P}_{m-n-1}) \\
&+2\,\varepsilon\,\,(-1)^n\,({P}_{m-n+2}+{P}_{m-n-2})\, \\
&+i\,\varepsilon\,[\,(-1)^n\,({P}_{m-n+3}-{P}_{m-n-3})+({P}_{m-n-1}-{P}_{m-n+1})] \\
=&(-1)^{n}\,2\,{P}_{m-n}\,(1+i\,+6\,\varepsilon\,+6\,i\,\varepsilon\,).
\end{array}
\end{equation*}
where the identity ${P}_{n}{P}_{m}+{P}_{n+1}{P}_{m+1}={P}_{n+m+1}$ and ${P}_{n+3}-{P}_{n-3}=14\,{P}_{n}$ are used \cite{I,Z,Z1}.
\end{proof}
\begin{thm} 
Let ${Q_P}_{n}$ be the dual-complex Pell quaternion.Then, we have the following identities
\begin{equation}\label{61}
\sum\limits_{s=1}^{n}{\,{{Q}_P}_{s}}=\frac{1}{4}[\,{{Q}_{PL}}_{n+1}-{{Q}_{PL}}_{1}\,],
\end{equation}
\begin{equation}\label{62}
\sum\limits_{s=0}^{p}{\,{{Q}_P}_{n+s}}=\frac{1}{4}[\,{{Q}_{PL}}_{n+p+1}-{{Q}_{PL}}_{n+1}\,],
\end{equation}
\begin{equation}\label{63}
\sum\limits_{s=1}^{n}{\,{{Q}_P}_{2s-1}}=\frac{1}{2}\,({{Q}_P}_{2n}-{{Q}_P}_{0}),
\end{equation}
\begin{equation}\label{64}
\sum\limits_{s=1}^{n}{\,{{Q}_P}_{2s}}=\frac{1}{2}\,[\,{{Q}_P}_{2n+1}-{{Q}_P}_{1}\,].
\end{equation} 
\end{thm}
\begin{proof}
(\ref{61}): 
\begin{equation*}
\begin{aligned}
  \sum\limits_{s=1}^{n}{{Q}_P}_{s}=&\sum\limits_{s=1}^{n}{{P}_{s}}+i\,\sum\limits_{s=1}^{n}{{P}_{s+1}}+\varepsilon\,\sum\limits_{s=1}^{n}{{P}_{s+2}}+i\,\varepsilon\,\sum\limits_{s=1}^{n}{{P}_{s+3}} \\
=&\frac{1}{2}[(P_{n}+{P_{n+1}}-P_{1}-P_{0})+i\,({P}_{n+1}+{P_{n+2}}-P_{2}-P_{1})\\
&+\varepsilon\,({{P}_{n+2}}+{P_{n+3}}-P_{3}-P_{2})+i\,\varepsilon\,({{P}_{n+3}}+{P_{n+4}}-P_{4}-P_{3})] \\
=&\frac{1}{2}\,[{Q_P}_{n}+{Q_P}_{n+1}-{Q_P}_{1}-{Q_P}_{0}] \\
=&\frac{1}{4}[\,{{Q}_{PL}}_{n+1}-{{Q}_{PL}}_{1}\,]\,. 
\end{aligned}
\end{equation*}
(\ref{62}): Hence,  we can write
\begin{equation*}
\begin{aligned}
 \sum\limits_{s=0}^{p}{{Q}_P}_{n+s}=&\sum\limits_{s=0}^{p}{{P}_{n+s}}+i\,\sum\limits_{s=0}^{p}{{P}_{n+s+1}}+\varepsilon\,\sum\limits_{s=0}^{p}{{P}_{n+s+2}}+i\,\varepsilon\,\sum\limits_{s=0}^{p}{{P}_{n+s+3}} \\
=&\frac{1}{2}[({{P}_{n+p+1}}+{{P}_{n+p}}-P_{n+1}-P_{n}) \\
&+i\,({P}_{n+p+2}+{P}_{n+p+1}-P_{n+2}-P_{n+1}) \\
&+\varepsilon\,({{P}_{n+p+3}}+{{P}_{n+p+2}}-P_{n+3}-P_{n+2}) \\
&+i\,\varepsilon\,({P}_{n+p+4}+{P}_{n+p+3}-P_{n+4}-P_{n+3})] \\
=&\frac{1}{2}[\,{{Q}_P}_{n+p+1}+{{Q}_P}_{n+p}-{{Q}_P}_{n+1}-{{Q}_P}_{n}\,] \\
=&\frac{1}{4}[\,{{Q}_{PL}}_{n+p+1}-{{Q}_{PL}}_{n+1}\,]\,. 
\end{aligned}
\end{equation*}
(\ref{63}) Hence,  we can write 
\begin{equation*}
\begin{aligned}
& \sum\limits_{s=1}^{n}{{Q}_P}_{2s-1}=\sum\limits_{s=1}^{n}{P}_{2s-1}+i\,\sum\limits_{s=1}^{n}{P}_{2s}+\varepsilon\,\sum\limits_{s=1}^{n}{P}_{2s+1}+i\,\varepsilon\,\sum\limits_{s=1}^{n}{P}_{2s+2} \\
&\quad \quad \quad \quad \quad =({P}_{1}+{P}_{3}+\ldots+{P}_{2n-1})+i({P}_{2}+{P}_{4}+\ldots+{P}_{2n}) \\
&\quad \quad \quad \quad \quad \quad +\varepsilon\,({P}_{3}+{P}_{5}+\ldots+{P}_{2n+1})+i\,\varepsilon\,({P}_{4}+{P}_{6}+\ldots+{P}_{2n+2}) \\
&\quad \quad \quad \quad \quad =\frac{1}{2}\,[({P}_{2n}-P_{0})+i\,({P}_{2n+1}-P_{1}) \\
&\quad \quad \quad \quad \quad+\varepsilon\,({P}_{2n+2}-P_{2})+i\,\varepsilon\,({P}_{2n+3}-P_{3})]  \\
&\quad \quad \quad \quad \quad =\frac{1}{2}\,[{P}_{2n}+i\,{P}_{2n+1}+\varepsilon\,{P}_{2n+2}+i\,\varepsilon\,{P}_{2n+3}] \\
&\quad \quad \quad \quad \quad-\frac{1}{2}\,[{P}_{0}+i\,{P}_{1}+\varepsilon\,{P}_{2}+i\,\varepsilon\,{P}_{3}]\\
&\quad \quad \quad \quad \quad =\frac{1}{2}\,({{Q}_P}_{2n}-{{Q}_P}_{0}).
\end{aligned}
\end{equation*}
(\ref{64}): Hence,  we obtain
\begin{equation*}
\begin{aligned}
 \sum\limits_{s=1}^{n}{{Q}_P}_{2s}=&\,
({P}_{2}+{P}_{4}+\ldots+{P}_{2n})+i\,({P}_{3}+{P}_{5}+\ldots+{P}_{2n+1})
\\ &+\varepsilon\,({P}_{4}+{P}_{6}+\ldots+{P}_{2n+2})+i\,\varepsilon\,({P}_{5}+{P}_{7}+\ldots+{P}_{2n+3}) \\
=&\frac{1}{2}\,[({P}_{2n+1}-P_{1})+i\,({P}_{2n+2}-P_{2})+\varepsilon\,({P}_{2n+3}-P_{3}) \\ 
&+i\,\varepsilon\,({P}_{2n+4}-P_{4})]  \\
=&\,\frac{1}{2}[\,{P}_{2n+1}+i\,{P}_{2n+2}+\varepsilon\,{P}_{2n+3}+i\,\varepsilon\,{P}_{2n+4}\,] \\ &-\frac{1}{2}[\,P_{1}+i\,P_{2}+\varepsilon\,P_{3}+i\,\varepsilon\,P_{4}\,]\\
 =&\,\frac{1}{2}\,[\,{{Q}_P}_{2n+1}-{{Q}_P}_{1}\,]\,. 
\end{aligned}
\end{equation*}
\end{proof}
\begin{thm} \textbf{Binet's Formula} 
Let ${Q_P}_{n}$ be the dual-complex Pell quaternion. For $n\ge 1$, Binet's formula for these quaternions is as follows:
\begin{equation}\label{65}
{Q_P}_{n}=\frac{1}{\alpha -\beta }\left( \,\hat{\alpha }\,\,{\alpha }^{n}-\hat{\beta \,}\,{\beta }^{n} \right)\,
\end{equation}
where
\begin{equation*}
\begin{array}{l}
\hat{\alpha }=1+i\,{\alpha}+\varepsilon\,{\alpha}^2+i\,\varepsilon\,{\alpha}^3,\,\,\,\,\, \alpha={1+\sqrt2}
\end{array}
\end{equation*}
and
\begin{equation*}
\begin{array}{l}
\hat{\beta }=1+i\,{\beta}+\varepsilon\,{\beta}^2+i\,\varepsilon\,{\beta}^3,\,\,\,\,\, \beta={1-\sqrt2}.
\end{array}
\end{equation*}
\end{thm}
\begin{proof} 
Binet's formula of the dual-complex Pell quaternion is the same as Binet's formula of the Pell quaternion \cite{I}.
\end{proof}
\begin{thm} \textbf{Cassini's Identity} 
Let ${Q_P}_{n}$ be the dual-complex Pell quaternion. For $n\ge 1$, Cassini's identity for ${Q_P}_{n}$ is as follows: 
\begin{equation}\label{66}
{Q_F}_{n-1}\,{Q_F}_{n+1}-{Q_F}_{n}^2=(-1)^{n}\,2\,(1+i\,+6\varepsilon\,+6\,i\,\varepsilon\,). 
\end{equation}
\end{thm}
\begin{proof}
(\ref{66}): By using (\ref{44}) we get
\begin{equation*}
{\begin{array}{rl}
{\,{{Q}_P}_{n-1}}\,{\,{{Q}_P}_{n+1}}\,-(\,{{Q}_P}_{n})^2=&
\,({P}_{n-1}{P}_{n+1}-{P}_{n}^2)+({P}_{n+1}^2-{P}_{n}{P}_{n+2}) \\
&-i\,({P}_{n+1}{P}_{n}-{P}_{n+2}{P}_{n-1}) \\
&+\varepsilon\,[-({P}_{n+2}{P}_{n}-{P}_{n+3}{P}_{n-1}) \\
&-({P}_{n}{P}_{n+2}-{P}_{n+1}{P}_{n+1}) \\
&+({P}_{n+1}{P}_{n+3}-{P}_{n+2}{P}_{n+2})\\
&+({P}_{n+3}{P}_{n+1}-{P}_{n+4}{P}_{n})] \\
&+i\,\varepsilon\,[-{P}_{n}{P}_{n+3}-{P}_{n+1}{P}_{n+2})\\
&-({P}_{n+3}{P}_{n}-{P}_{n+4}{P}_{n-1})\\
&-({P}_{n+2}{P}_{n+1}-{P}_{n+3}{P}_{n})\\
&-({P}_{n+1}{P}_{n+2}-{P}_{n+2}{P}_{n+1})] \\
=&(-1)^{n}\,2\,(0+0i+6j+3\,ij) \\
=&(-1)^{n}\,2\,(1+i\,+6\varepsilon\,+6\,i\,\varepsilon\,). 
\end{array}}
\end{equation*} 
where the identities of the Pell numbers ${P}_{m}{P}_{n+1}-{P}_{m+1}{P}_{n}=(-1)^{n}{P}_{m-n}$\, is used \,\cite{I}.  
\end{proof}
\begin{thm} \textbf{Catalan's Identity} 
Let ${Q_P}_{n}$ be the dual-complex Pell quaternion. For $n\ge 1$, Catalan's identity for ${Q_P}_{n}$ is as follows: 
\begin{equation}\label{67}
{Q_P}_{n}^2-{Q_P}_{n+r}\,{Q_P}_{n-r}=(-1)^{n-r}\,{P}_{r}^2\,\,2\,(1+i\,+6\,\varepsilon\,+6\,i\,\varepsilon\,).
\end{equation}
\end{thm}
\begin{proof} 
(\ref{67}): By using (\ref{44}) we get 
\begin{equation*}
\begin{array}{rl}
{Q_P}_{n}^2-{Q_P}_{n+r}\,{Q_P}_{n-r}=&[\,({P}_{n}^2-{P}_{n+r}\,{P}_{n-r})-({P}_{n+1}^2-{P}_{n+r+1}\,{P}_{n-r+1})\,] \\ 
&+i\,[\,({P}_{n}\,{P}_{n+1}-{P}_{n+r}\,{P}_{n-r+1}) \\&\quad \quad+({P}_{n+1}\,{P}_{n}-{P}_{n+r+1}\,{P}_{n-r})\,]\\ &+\varepsilon\,[\,({P}_{n+2}\,{P}_{n}-{P}_{n+r+2}\,{P}_{n-r})\\&\quad \quad+({P}_{n}\,{P}_{n+2}-{P}_{n+r}\,{P}_{n-r+2}) \\ &\quad \quad-({P}_{n+1}\,{P}_{n+3}-{P}_{n+r+1}\,{P}_{n-r+3}) \\ &\quad \quad-({P}_{n+3}\,{P}_{n+1}-{P}_{n+r+3}\,{P}_{n-r+1})\,]\\
 &+i\,\varepsilon\,[\,({P}_{n}\,{P}_{n+3}-{P}_{n+r}\,{P}_{n-r+3})\\&\quad \quad+({P}_{n+3}\,{P}_{n}-{P}_{n+r+3}\,{P}_{n-r}) \\ &\quad \quad+({P}_{n+1}\,{P}_{n+2}-{P}_{n+r+1}\,{P}_{n-r+2}) \\ &\quad \quad+({P}_{n+2}\,{P}_{n+1}-{P}_{n+r+2}\,{P}_{n-r+1})\,] \\
&=(-1)^{n-r})\,2\,{P}_{r}^2\,(1+i\,+6\,\varepsilon\,+6\,i\,\varepsilon\,) 
\end{array}
\end{equation*}
where the identities of the Pell numbers ${P}_{m}{P}_{n}-{P}_{m+r}{P}_{n-r}=(-1)^{n-r}{P}_{m+r-n}\,P_{r}$\, and \,$P_{n}\,P_{n}-P_{n-r}\,P_{n+r}=(-1)^{n-r}\,P_{r}^2$ are used\,\cite{I}. 
\end{proof}

\section{Conclusion} 
In this study, a number of new results on dual-complex Pell quaternions were derived. Quaternions have great importance as they are used in quantum physics, applied mathematics, quantum mechanics, Lie groups, kinematics and differential equations. 

This study fills the gap in the literature by providing the dual-complex Pell quaternion using definitions of the dual-complex number, dual-complex Pell number and Pell quaternion \cite{L}.


\begin{thebibliography}{99}
\bibitem{U} Kotelnikov, A.: Screw calculus and some of its applications to geometry and mechanics. Annals of Imperial University of Kazan.
\end{thebibliography}
\end{document}